\documentclass[11pt]{amsart}

\usepackage{amssymb,amsfonts,lineno,amsthm,amscd,stmaryrd,dsfont,esint}

\usepackage{tikz}
\usetikzlibrary{decorations.pathmorphing}

\usepackage[margin=1in]{geometry}
\usepackage{xcolor}

\theoremstyle{plain}
\newtheorem{thm}{Theorem}

\newtheorem{lem}{Lemma}
\newtheorem{prop}[lem]{Proposition}

\theoremstyle{definition}

\numberwithin{lem}{section}
\numberwithin{equation}{section}

\newcommand{\ep}{\varepsilon}
\newcommand{\R}{\mathbb{R}}
\newcommand{\C}{\mathbb{C}}
\newcommand{\Rd}{\mathbb{R}^\d}

\renewcommand{\d}{d}

\renewcommand{\P}{\mathbb{P}}

\newcommand{\I}{\mathcal I}
\renewcommand{\Re}{\mathrm{Re}}
\newcommand{\muup}{\widehat\mu}
\newcommand{\indic}{\mathds{1}}
\renewcommand{\H}{H^{\muup}}
\newcommand{\nab}{\nabla}
\DeclareMathOperator{\dist}{dist}
\DeclareMathOperator{\supp}{supp}

\begin{document}

\title[A Constrained optimization problem for the Ginibre ensemble]{Remarks on a constrained optimization problem for the Ginibre ensemble}

\author[S. N. Armstrong]{Scott N. Armstrong}
\address{CEREMADE (UMR CNRS 7534), Universit\'e Paris-Dauphine, Paris, France}
\email{armstrong@ceremade.dauphine.fr}

\author[S. Serfaty]{Sylvia Serfaty}
\address{Laboratoire Jacques Louis Lions, UPMC Paris 6, France and The Courant Institute of Mathematical Sciences, New York University, USA\\}
\email{serfaty@ann.jussieu.fr}

\author[O. Zeitouni]{Ofer Zeitouni}
\address{Faculty of Mathematics, Weizmann Institute of Science, Rehovot 76100, Israel and School of Mathematics, University of Minnesota, MN 55455, USA}
\email{zeitouni@math.umn.edu}

\date{April 1, 2013}

\begin{abstract}

We study the limiting distribution of the eigenvalues of the Ginibre ensemble conditioned on the event that a certain proportion lie in a given region of the complex plane. Using an equivalent formulation as an obstacle problem, we describe the optimal distribution and some of its monotonicity properties.
\end{abstract}\maketitle

\maketitle

\section{Introduction and statement of results}

With high probability, a square $N$-by-$N$ matrix with complex independent, 
standard
Gaussian entries (i.e., the {\it Ginibre ensemble}) has eigenvalues evenly spread in the ball of radius $\sqrt N$. More precisely, after multiplying the matrix by $N^{-1/2}$ and letting $N\to \infty$, one finds that the spectral measure converges (weakly in the sense of measures) to the (suitably normalized) uniform measure on the unit disk $D \subseteq \C$:
\begin{equation*}\label{}
\frac1N \sum_{i=1}^N \delta_{z_i} \rightharpoonup \mu_0
\end{equation*}
where the $z_i$'s are the eigenvalues of the matrix, $\delta_{z}$ denotes the Dirac mass at $z\in \C$ and
\begin{equation}\label{mu0}
d\mu_0 : = \frac1{\pi} \indic_{D}\,dx.
\end{equation}
This statement is called the \emph{circular law} and goes back to Ginibre~\cite{G}.

Hiai and Petz~\cite{HP}
proved, in addition, that the law of spectral measure of the eigenvalues obeys a large deviations principle with speed $N^2$ and rate function
\begin{equation}\label{e.I}
\I[\mu] := -\int_{\C\times \C}  \log |x-y| \, d\mu(x) \, d\mu(y) + \int_{\C} |x|^2 \, d\mu(x).
\end{equation}
(See Ben Arous and Zeitouni~\cite{BAZ} for the case of real Gaussian entries.)
Roughly, this means that if $\mathcal A \subseteq \mathcal P(\C)$, then
\begin{equation*}\label{}
\P \left[ \frac1N \sum_{i=1}^N \delta_{z_i} \in \mathcal A \right] \simeq \exp\left( -N^2 (\min_{\mathcal A} \I- \min_{\mathcal P(\C) }  \I) \right).
\end{equation*}
Of course, the unique minimizer of $\I$ over the set $\mathcal P(\C)$ of probability measures on $\C$ is the circular
law $\mu_0$. In view of this result, in order to understand the likely arrangement of eigenvalues after conditioning on a certain rare event (e.g., 
having a ``hole" in $D$ without a significant number of
eigenvalues) it suffices to find the minimizer of $\I$ on the event.

Related random matrix models are the {\it Gaussian unitary ensemble} (GUE) and {\it Gaussian orthogonal ensemble} (GOE), in which the matrices are constrained to be Hermitian and real symmetric, respectively, and thus have real eigenvalues. The analogue of the circular law is Wigner's \emph{semicircle law}:
\begin{equation*}\label{}
\frac1N \sum_{i=1}^N \delta_{x_i} \rightharpoonup \nu_0 := \indic_{ |x| \leq 2} \sqrt{4-x^2}\,dx.
\end{equation*}
Here $\nu_0$ minimizes $\I$ over the set $\mathcal P(\R)$ of probability measures on the real line. The corresponding large deviations principle in this setting was proved in the seminal paper by Ben Arous and Guionnet~\cite{BAG}.

In this context, Majumdar, Nadal, Scardicchio and Vivo~\cite{MNSV1,MNSV} examined the rare event that a given proportion $p\in [0,1]$ of the eigenvalues must lie on interval $[a,\infty)$. Determining the likely configuration of the eigenvalues as $N\to \infty$, conditioned on this rare event, is equivalent to minimizing $\I$ (over probability measures on $\R$) under the constraint that $\mu( [a,\infty)) \geq p$. Explicit formulas are given in~\cite{MNSV} for the optimal distribution $\mu$. 
 A nontrivial situation occurs when $p> \nu_0([a,\infty))$, in which case they find that $\mu$ has an $L^1$ density $\varphi(x)$ which tends to infinity as $x \to a+$ and vanishes in an interval $[a-\ep,a)$ for some $\ep > 0$.

In this paper we examine the analogue of this question for the Ginibre ensemble.
We consider an open subset $U \subseteq \C$  with $C^{1,1}$ boundary 
{such that the boundary of $U \cap D$ is locally Lipschitz,} and we study the unlikely event in which the random matrix has too many eigenvalues in $U$. To avoid trivial situations we assume that $D \setminus U$ is nonempty, fix $\mu_0(U)< p \leq 1$ and minimize $\I$ subject to the constraint that
\begin{equation}\label{e.const}
\mu\!\left(\overline U\right) \geq p.
\end{equation}
Due to the convexity of the constraint and the strict convexity of the
functional $\I$, there is a unique minimizer of $\I$ over
\begin{equation*}\label{}
\mathcal C:= \{ \mu \in \mathcal P(\C) \, : \, \eqref{e.const} \ \mbox{holds} \},
\end{equation*}
which we denote by $\muup$.
Note that if $p=1$, then $\mathcal C$ is the set of probability measures with a ``hole" in $\C\setminus \overline U$. 

In contrast to the results for the one dimensional model studied in~\cite{MNSV}, our analysis typically
does not lead to explicit formulas. Rather, we study certain qualitative properties of the minimizing measure $\muup$.
(An exception is the case where $p=1$ and the set $U$ is a half-space parallel
to the imaginary axis; see Section \ref{sec-example} for that setup.)

Here is the statement of the main results. (See Figure~\ref{F} for an illustration of the results.)
\begin{thm} \label{T}
{Assume that $U\subseteq \C$ is open, $\partial U$ is locally $C^{1,1}$ and $\partial (U\cap D)$ is locally Lipschitz.} Then measure $\muup$ can be decomposed as
\begin{equation*}\label{}
d\muup = \frac1\pi \indic_{V}\,dx + d\muup_S,
\end{equation*}
where $V\subseteq \C$ is bounded and open and 
\begin{equation*}\label{}
d\muup_S = g\, d\mathcal H^1\vert_{\partial U},
\end{equation*}
where $\mathcal H^1$ denotes the one-dimensional Hausdorff measure and $g\geq 0$ belongs to 
$L^2(\partial U)$. Moreover,
\begin{equation}\label{e.sing}
\partial U \cap D \subseteq \supp\muup_S,
\end{equation}
\begin{equation} \label{e.gap}
  \overline {\big(V\setminus  U \big)}\cap \partial U = \varnothing,
\end{equation}
\begin{equation}\label{e.contract}
  V\setminus U\subseteq D, \quad \mbox{and moreover, if} \ \ U\not\subseteq D  \ \ \mbox{then}  \ \ \overline{V}\setminus \overline{U} \subseteq D
\end{equation}
and finally
\begin{equation}\label{e.expand}
{\overline{U \cap D}} \subseteq V \cup \supp \muup_S.
\end{equation}
\end{thm}

The fact that, at least in the case that $\partial U \cap D \neq \emptyset$, the singular part of $\muup$ is nonzero is in contrast to the one-dimensional model~\cite{MNSV} which has a minimizing measure absolutely continuous with respect to Lebesgue measure on $\R$. On the other hand, the qualitative picture portrayed by~\eqref{e.contract} and~\eqref{e.expand} -- an expansion of the support of the measure in $U$, the contraction of it in the complement of~$U$ and the appearance of a gap on the outside of~$\partial U$ along the support of $\muup_S$ -- is the same as that of~\cite{MNSV}. 

\def\thesetU{plot [smooth cycle] coordinates { (3.6,5.17) (6,2.8) (6.2,-2.5) (4.95,-3.9) (2.85,-2.7) (2.9,-0.95) (4,0.2) (3.5,1.35) (2.15,1) (1.1,1.7) (1.5,3)}}
\def\thedisk{(0,0) circle (5.17)}
\def\thediskk{ (0.2,0) circle (5.25)}
\def\thediskkk{(-.2,0) circle (4.8)}
\def\VoutsideU{plot [smooth cycle] coordinates { (1.442,4.511) 
(1.042,4.636) (0.634,4.727) (0.218,4.782) (-0.400,4.800) (-0.618,4.782) (-1.034,4.727) (-1.442,4.636) (-1.842,4.511) (-2.229,4.350) ( -2.600,4.157) (-2.953,3.932) (-3.285,3.677) (-3.594,3.394) (-3.877,3.085) (-4.132,2.753) (-4.357,2.400) (-4.550,2.029)(-4.711,1.642) (-4.836,1.242)(-4.927,0.834) (-4.982,0.418) (-5.000,0.000)(-4.982,-0.418)(-4.927,-0.834)(-4.836,-1.242)(-4.711,-1.642)(-4.550,-2.029)(-4.357,-2.400)(-4.132,-2.753)(-3.877,-3.085)(-3.594,-3.394)(-3.285,-3.677)(-2.953,-3.932)(-2.600,-4.157)(-2.229,-4.350)(-1.842,-4.511)(-1.442,-4.636)(-1.034,-4.727)(-0.618,-4.782)(-0.200,-4.800) (0.218,-4.782)(0.634,-4.727) (1.042,-4.636)(1.442,-4.511)
(1.829,-4.350) (2.5,-4) (3.1,-3.4)  (2.41,-2.65) (2.47,-0.95)(3.15,0.22) (3.01,.735) (1.9,0.65) (0.75,1.6) (1.02,3) (1.8,4.2) } }
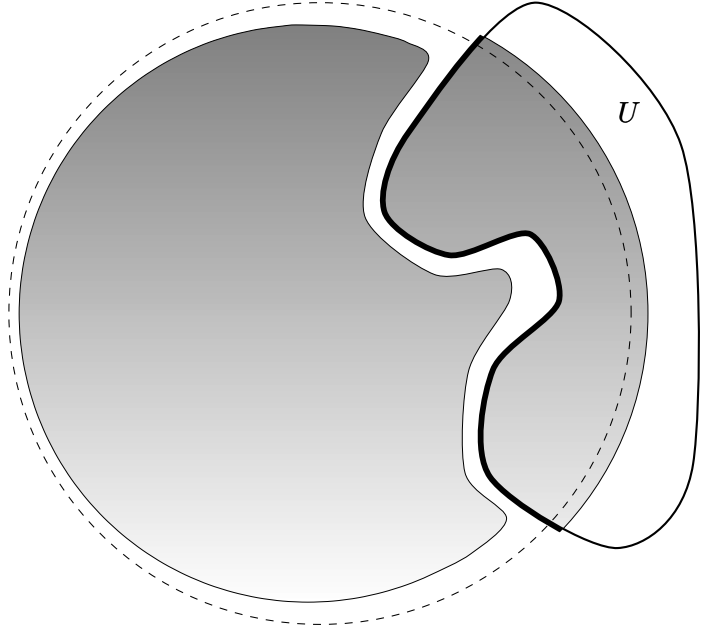
\begin{figure}
\centering
\begin{tikzpicture}[scale=0.8]
\begin{scope}
\clip \thediskk;
\draw[line width=1.25mm] \thesetU;
\end{scope}
\begin{scope}
\clip \thesetU;
\shade \thediskk;
\draw[ultra thin] \thediskk;
\end{scope}
\shade \VoutsideU;
\draw[ultra thin] \VoutsideU;
\draw[dashed] \thedisk;
\draw[thick] \thesetU;
\draw (5.5,3.35) node[left] 
      {$U$};
 \draw (5.5,3.35) node[left] 
      {$U$};
\end{tikzpicture}
\caption{An illustration of Theorem~
\ref{T} in the case $p$ is slightly larger than $\mu_0(U)$ and $U \not\subseteq D$. The dashed line is the unit ball and the shaded area is the support of the measure $\muup$. The thick part of the boundary of $\partial U$ represents the support of $\hat \mu_S$.}
\label{F}
\end{figure}

The proof of Theorem~\ref{T} is based on interpreting the minimization problem as an \emph{obstacle problem}; this connection should be
well-known, but we are not aware of an explicit reference. The advantage of writing the problem this way is that it allows us to use the classical free-boundary regularity theory of Caffarelli~\cite{C} combined with some maximal principle arguments such as Hopf's Lemma. 

\subsection*{Acknowledgements}

We thank Gilles Wainrib and Amir Dembo
for bringing this problem to our attention, and Satya Majumdar for enlighting 
discussions and a reference to \cite{MNSV}.
SS was supported by a EURYI award. OZ was partially supported by NSF grant DMS-1203201, the Israel Science Foundation and the Herman P. Taubman chair of Mathematics at the Weizmann Institute. 

\section{The proofs}

\subsection{Formulation as an obstacle problem}
We first characterize the measure $\muup$ in terms of an obstacle problem. The solution of the obstacle problem turns out to be the Newtonian potential generated by~$\muup$, which we denote by 
\begin{equation}\label{defH}
\H(x) := -\int_{\R^2} \log |x-y|\, d\muup(y).
\end{equation}
(We henceforth identify $\R^2$ with $\C$.) We show that $\H$ satisfies the variational inequality
\begin{equation}\label{e.vi}
\forall v\in \mathcal K,\qquad  \int_{\R^2} \nabla \H(y) \cdot \nabla \big( v - \H \big)(y)\, dy \geq 0,
\end{equation}
where 
\begin{equation*}\label{}
\mathcal K:= \left\{ v\in H^1_{\mathrm{loc}} (R^2)\,:\,   v-\H \ \ \mbox{has bounded support in} \ \R^2 \ \ \mbox{and} \ \ v \geq \psi \ \mbox{q.e. in} \ \R^2 \right\}
\end{equation*}
and the \emph{obstacle} is the function $\psi$ given by
\begin{equation}\label{e.psi}
  \psi(x):= \begin{cases}  \frac12 (c_1 - |x|^2 ) & x\in \overline{U}\\ 
    \frac12 (c_2-|x|^2) & x\not\in \overline{U}, \end{cases}
\end{equation}
and $c_1>c_2 \geq -\infty$ are constants determined below. Note that~\eqref{e.vi} is equivalent to the statement that, for every $R>0$, $\H$ is the unique solution of the following minimization problem:
\begin{equation}\label{e.optp}
\min \left\{ \int_{B_R} \big|\nabla v(y)\big|^2\, dy \,:\, v \in H^1(B_R)\ \mbox{such that} \  v-\H\in H^1_0(B_R) \ \mbox{and} \ v\geq \psi \ \mbox{in} \ B_R \right\}.
\end{equation}
Solutions of the variational inequality are said to be solutions of the \emph{obstacle problem}
\begin{equation}\label{e.obst}
\min\left\{ -\Delta \H, \H - \psi \right\} = 0 \quad \mbox{in} \ \R^2.
\end{equation}
That is, the precise interpretation of~\eqref{e.obst} is that $\H$ is a solution of~\eqref{e.vi}. The reader can check by an integration by parts that this notion is consistent with the classical interpretation of~\eqref{e.obst} in the case that $\H$ is smooth. For general reference on the obstacle problem, see~\cite{C,KS}.

To prove~\eqref{e.obst}, we observe that $-\Delta \H = 2\pi  \muup \ge 0$ and in particular $\H$ is superharmonic and harmonic in the complement of the support of $\muup$. It remains to show that $\H \geq \psi$ in $\R^2$ and $\H = \psi$ on the support of $\muup$. This is the statement of Lemma~\ref{lem2.1}, below. In Lemma~\ref{l.vi}, we prove that $\H$ satisfies~\eqref{e.vi}.

In what follows, ``q.e." is short for \emph{quasi-everywhere} which means ``except possibly in of a set of capacity zero."
A set $K$ is of \emph{capacity zero} if there does not exist 
a probability measure $\nu$ with support in a compact subset of 
$K$ such that $\iint -\log |x-y|\, d\nu(x)\, d\nu(y)$ 
is finite (c.f.~\cite{ST}). In particular, a set of capacity zero 
has zero Lebesgue measure and we note that a measure 
$\mu$ satisfying $\I[\mu] < +\infty$ does not charge sets of zero capacity. 

We first observe that $\muup (\overline U) =p$. If, on the contrary, $\muup(\overline U) >p$, then for small enough $t>0$, we have $(1-t) \muup + t \mu_0 \in \mathcal C$. The strict convexity of $\I$ then contradicts the minimality of $\muup$:
$$\I [ (1-t) \muup + t\mu_0] < (1-t)  \I [\muup] + t \I [\mu_0] \le \I [\muup].$$

\begin{lem}\label{lem2.1}
There exist $c_1>c_2 \geq -\infty$ with $c_2=-\infty$ if $p=1$, such that
\begin{equation}
\label{2.1}
2 \H(x) + |x|^2 \ge c_1 \quad \mbox{q.e. in} \ \overline U \quad \mbox{and} \quad  2\H(x) +|x|^2 =c_1 \quad  \text{q.e.} \ \text{in} \ \supp (\muup) \cap \overline U
\end{equation}
and
\begin{equation}
\label{2.3}
2\H(x) + |x|^2 \ge c_2 \quad \mbox{q.e. in} \ \C\setminus \overline U  \quad   \mbox{and} \quad 2\H(x) + |x|^2 = c_2 \quad  \text{q.e.} \ \text{in}  \  \supp(\muup)\setminus \overline U.
\end{equation}
\end{lem}
\begin{proof}
We argue by the standard method of variations (one could also proceed by computing the convex dual of the optimization problem). Select~$\nu\in \mathcal C$ such that $\I[\nu]<+\infty$ and observe that, for every $t \in [0,1]$,
\begin{equation}\label{}
\I \left[(1-t)\muup + t \nu\right] \ge \I [\muup].
\end{equation}
Expanding in powers of $t$ and letting $t \to 0$, we discover that
\begin{equation}\label{ineq}
\int_\C \left( 2\H(x) + |x|^2 \right) \, d (\nu - \muup) (x)   \ge 0
\end{equation}
for every $\nu \in \mathcal C$ for which $\I[\nu] < \infty$.

We claim that
\begin{equation}\label{ineq1}
 2\H(x)+ |x|^2   \ge \frac1p \int_{\overline{U}} \left(2\H(x) + |x|^2 \right) \, d\muup(x)=: c_1\quad \text{q.e. in } \ \overline U.
\end{equation}
If~\eqref{ineq1} is false, then there
exists a compact
set $E\subseteq \overline U$ of positive capacity on 
which the opposite inequality holds. By the 
characterization of positive capacity, there exists a measure $\nu_E$ 
of mass~$p$ supported in $E$ with $\iint -\log |x-y|\, d\nu_E(x)\, d\nu_E(y) <\infty$. We obtain a contradiction to~\eqref{ineq} by taking $\nu$ to be the sum of $\nu_E$ and the restriction of $\muup$ to $\C \setminus \overline U$. This establishes~\eqref{ineq1}, and integrating it against $\muup$ (which does not charge sets of zero capacity since it has finite energy) yields~\eqref{2.1}.

If $p=1$ we stop the reasoning here and all the claims are true since \eqref{2.3} holds with $c_2=-\infty$. If $p<1$, 
we obtain~\eqref{2.3} by a similar argument as above, applying~\eqref{ineq} to a measure $\nu$ which coincides with~$\muup$ in~$\overline {U}$ and has mass~$1-p$ in~$\C \setminus \overline{U}$, resulting with
$$c_2:=\frac1{1-p} \int_{\C\setminus 
  {\overline{U}}} \left(2\H(x) + |x|^2 \right) \, d\muup(x).$$ 

It remains to prove that $c_1>c_2$. By strict convexity of $\I$ and the fact that $\I[ \muup] > \I[\mu_0]$, we have
 \begin{equation*}\label{}
\left. \frac{d}{dt} \right\vert_{t=0} \I \left[ (1-t) \muup + t\mu_0 \right] <0.
\end{equation*}
This says, by the same computation as above, that
\begin{equation*}\label{}
\int_\C \left(2\H (x) + |x|^2 \right) \, d(\mu_0- \muup)(x)<0.
\end{equation*}
Using~\eqref{2.1} and~\eqref{2.3}, we obtain
\begin{equation*}\label{}
c_1 \mu_0(U) + c_2 \mu_0(\C \setminus U) < p c_1+ (1-p) c_2.
\end{equation*}
Since $p> \mu_0(U)$ by assumption, we deduce that $c_1>c_2$.
\end{proof}

In the special case $p=1$, the result in Lemma~\ref{lem2.1} 
is actually a characterization.
\begin{lem}
  \label{lem2.1p=1}
  Assume $p=1$ and that $\mu\in \mathcal C$ satisfies
  \begin{equation}
\label{2.1p=1}
2 H^\mu(x) + |x|^2 \ge c \quad \mbox{q.e. in} \ \overline U \quad \mbox{and} \quad  2H^\mu(x) +|x|^2 =c \quad  \text{q.e.} \ \text{in} \ 
\supp (\mu) \cap \overline U\,,
\end{equation}
for some $c\in \R$.
Then, $\mu=\muup$.
\end{lem}
\begin{proof}
  Assume not, and set, for $t\in (0,1)$, 
  $$\mu_t:=t\mu+(1-t)\muup.$$
  We have
  \begin{align}
    \label{eq2.3p=1}
  \I\left[\mu_t\right]&= 
  t
  \int (H^\mu(x)+|x|^2) \, d\mu_t(x)+(1-t)
  \int (\H(x)+|x|^2)\, d\mu_t(x)\nonumber \\
  &=\frac{t}{2}
  \int (2H^\mu(x)+|x|^2) \, d\mu_t(x)+\frac{(1-t)}{2}
  \int (2\H(x)+|x|^2) d\mu_t(x)+\frac12\int |x|^2 \, d\mu_t(x) \nonumber \\
 &\geq 
 \frac{t}{2} \left(c+\int |x|^2 \, d\mu(x)\right)+
 \frac{(1-t)}{2}\left(c_1+\int |x|^2 \, d\muup(x)\right)\,,
 \end{align}
where~\eqref{2.1} and~\eqref{2.1p=1} were used in the inequality. On the other hand, integrating~\eqref{2.1} and~\eqref{2.1p=1}  with respect to $\mu$ and $\muup$ gives
\begin{equation*}
\I[\mu]=\frac{1}{2}\left( c+\int |x|^2 d\mu(x)\right) \quad \mbox{and} \quad 
\I[\muup]=\frac{1}{2} \left( c+\int |x|^2 d\muup(x) \right).
\end{equation*}
Substituting in~\eqref{eq2.3p=1} then gives
\begin{equation*}\label{}
\I[\mu_t]\geq t \I[\mu]+(1-t) \I[\muup],
\end{equation*}
which contradicts the strict convexity of $\I$ unless $\mu=\muup$.
\end{proof}

\begin{lem}
\label{l.vi}
The function $\H$ is the unique element of $\mathcal K$ which satisfies~\eqref{e.vi}.
\end{lem}
\begin{proof}
First, we show that $\nabla\H$ is in $L^2_{\mathrm{loc}} (\R^2;\, \R^2)$.
To see this first note that if $\rho$ is a compactly supported Radon measure on $\R^2$ with $\int \, d\rho=0$ satisfying $\iint - \log |x-y| \, d\rho(x)_+\, d\rho(y)_+ <+\infty$, $\iint - \log |x-y| \, d\rho_-(x) \, d\rho_-(y) <+\infty$  and $H^\rho$ is the potential generated by $\rho $, defined as usual  by $H^\rho (x):= - \int \log |x-y|\, d \rho(y)$, then  we have 
\begin{equation}\label{e.hro}
\int_{\R^2} \int_{\R^2} - \log |x-y|\, d\rho(x)\, d\rho(y)= \frac{1}{2\pi}\int_{\R^2} |\nabla H^\rho(x) |^2\,dx.
\end{equation}
See the proof of \cite[Lemma 1.8]{ST} where, more precisely,  it is shown that 
\begin{equation*}
\int_{\R^2} \int_{\R^2} - \log |x-y| \, d\rho(x) \, d\rho(y)= \frac{1}{2\pi} \int_{\R^2} \left(\int_{\R^2} \frac{1}{|x-y|} \, d \rho(y)\right)^2 \, dx.
\end{equation*}
Then~\eqref{e.hro} follows, since $\int_{\R^2} \frac{(x-y) \d \rho(y)}{|x-y|^2}= - \nab \H(x)$ in the distributional sense.

Next, consider  $\mu_0$  (defined in \eqref{mu0}) and its associated potential $H^{\mu_0}$, which belongs to $C^{1}(\R^2)$ by direct computation. The measure $\rho:=\muup- \mu_0$ is a compactly supported Radon measure with~$\int d\rho =0$. In addition, since $\I[\muup]<\infty$ and $\I [\mu_0]<\infty$, it satisfies the desired assumptions. Thus \eqref{e.hro} holds, which proves that $\nabla \H - \nab H^{\mu_0}$ is in $L^2(\R^2;\,\R^2)$.
Since $\nab H^{\mu_0}$ is in $L^\infty_{\mathrm{loc}}(\R^2;\,\R^2)$,  it follows that $\nabla \H \in L^2_{\mathrm{loc}}(\R^2;\,\R^2)$, as claimed.

We next turn to the proof that $\H$ satisfies \eqref{e.vi}. Pick $v\in \mathcal K$ and set $\varphi:=v-\H$. Observe that, in the case that $\varphi$ is smooth and compactly supported, we have
\begin{equation}\label{ineg}
\int_{\R^2} \nab \H \cdot \nab \varphi = 2 \pi \int_{\R^2} \varphi \, d\muup \ge 0.
\end{equation}
The last inequality of~\eqref{ineg} is true because $\varphi = v- \H \ge 0$ q.e.~in $\supp (\muup)$ due to the fact that $\H=\psi$ q.e.~in $\supp (\muup)$ and $ v \ge \psi $ q.e.~in~$\R^2$. The conclusion \eqref{ineg} follows since $\muup$ does not charge sets of capacity zero, and this proves that~\eqref{e.vi} is satisfied under the additional assumption that~$\varphi$ is smooth and compactly supported.

To obtain~\eqref{e.vi} for general $v\in \mathcal K$, we just need to show that the subset of $\mathcal K$ consisting of $v$ for which $v-\H$ is smooth is dense in $\mathcal K$ with respect to the topology of $H^1$. To see this, we fix $v\in \mathcal K$ and select $R>1$ so large that $v-\H$ is supported in $B_{R/2}$. Then we define
\begin{equation*}\label{}
v_{\ep,\delta} := \H + (v-\H)\ast \eta_\ep + \delta \chi_{R},
\end{equation*}
 where $\ep,\delta > 0$, $\eta_\ep$ is the standard mollifier and $\chi_R$ is a smooth function supported in $B_{2R}$ with $0\leq \chi_{R} \leq1$ and $\chi \equiv 1$ in $B_R$. It is clear that $v_{\ep,\delta}-\H$ is smooth. If $\ep > 0$ is sufficiently small relative to $\delta$, then $v_{\ep,\delta} > \psi$ and hence $v_{\ep,\delta} \in \mathcal K$. Finally, we note that if $\ep > 0$ is small enough relative to $\delta$, then $\|v_{\ep,\delta} - v\|_{H^1} < C\delta$. This concludes the proof of~\eqref{e.vi}.

The uniqueness of the solution to \eqref{e.vi} is standard: if $v$ is another solution, then first \eqref{e.vi} holds, and second, testing \eqref{e.vi} (this time holding for $v$)  with $\H$ yields 
$$\int_{\R^2} \nab v \cdot \nab (\H - v) \ge 0, $$
Adding the two relations yields 
$\int_{\R^2} \nab (\H-v) \cdot \nab (v - \H) \ge 0$
which immediately implies that $\H-v=0 $ in $ H^1(\R^2)$.
\end{proof}

\subsection{{Proof of Theorem~\ref{T}}}

We now present the proof of the main result.

\begin{proof}[{Proof Theorem~\ref{T}}]

The variational characterization provided by Lemma~\ref{l.vi} asserts that $\H$ is a ``solution of the obstacle problem with obstacle $\psi$" (with $\psi$ defined in~\eqref{e.psi}) and allows us to use the classical regularity theory for the obstacle problem. In particular, we see that, away from $\partial U$, $\H$ is locally~$C^{1,1}$ since $\psi$ is smooth there. This optimal regularity for the obstacle problem is an old result of Frehse~\cite{F}; for a proof, see also~\cite[Theorem~2]{C}.

In neighborhoods of $\partial U$, in which $\psi$ has a jump discontinuity, we cannot expect $\H$ to be so regular, in general. However, a result of Frehse and Mosco~\cite[Theorem 3.2]{FM}  implies that $\H$ is in fact H\"older continuous in a neighborhood of~$\partial U$. To verify that this result applies, we have to check that $\psi$ verifies the ``unilateral H\"older condition" summarized in lines (3.6) and (3.7) of~\cite{FM}, which in our case comes down to a smoothness condition on $\partial U$. Indeed, it is easy to check that all we need for the obstacle $\psi$ in~\eqref{e.psi} to satisfy this condition is that $\partial U$ satisfies an exterior cone condition, so our assumption that $\partial U$ is $C^{1,1}$ certainly suffices. We conclude that $\H \in C^{0,\alpha}_{\mathrm{loc}}(\R^2)$ for some $\alpha>0$. In particular,~$\H$ is continuous.

Since~$\H$ is continuous and~$\muup$ does not charge sets of capacity zero, we may upgrade~Lemma~\ref{lem2.1} by removing the ``q.e." qualifier from~\eqref{2.1} and~\eqref{2.3}. We have:
\begin{equation}
\label{2.1up}
2 \H(x) + |x|^2 \ge c_1 \quad \mbox{in} \ \overline U \quad \mbox{and} \quad  2\H(x) +|x|^2 =c_1  \quad \mbox{in} \ \supp (\muup) \cap \overline U,
\end{equation}
and
\begin{equation}
\label{2.3up}
2\H(x) + |x|^2 \ge c_2 \quad \mbox{in} \ \R^2 \setminus \overline U  \quad   \mbox{and} \quad 2\H(x) + |x|^2 = c_2 \quad \mbox{in}  \  \supp(\muup)\setminus \overline U.
\end{equation}

Define $V$ to be the interior of the support of $\muup$. 
Observe that $V$ is bounded because $\muup$ is compactly supported 
since $\I[\muup]<\infty$ and $|x|^2-\int \log|x-y|\,d\muup(y)\to {\infty}$
as $|x|\to\infty$. According to~\eqref{defH},~\eqref{2.1up}~\eqref{2.3up} and the fact that $\H$ is $C^{1,1}_{\mathrm{loc}}=W^{2,\infty}_{\mathrm{loc}}$ in both $U$ and $\R^2\setminus \overline{U}$, we find that
\begin{equation}\label{deltah}
\muup = -\frac{1}{2\pi}\Delta \H =  \frac{1}{\pi}\indic_V \quad \mbox{in} \ U \cup (\R^2 \setminus \overline{U}).
\end{equation}
Let $\muup_{\mathrm{reg}}$ denote the Lebesgue measure on $V$ multiplied by $1/\pi$ and set $\muup_S:= \muup - \muup_{\mathrm{reg}}$. It is clear by~\eqref{deltah} that $\muup_S$ is supported on $\partial U$.

Since $\H$ is continuous and above the obstacle function $\psi = \frac12(c_1-|x|^2)$ in $\overline U$, it must be strictly larger than $\frac12(c_2-|x|^2) = \psi$ in a neighborhood of $\overline U$. In other words, there is a gap between~$\partial U$ and~$V\setminus U$.

We next argue that $\H\in C^{0,1}_{\mathrm{loc}}(\R^2)$, which amounts to proving that $\H$ is Lipschitz continuous in a neighborhood of $\partial U$. To show this, we introduce a modification~$\hat\psi$ of~$\psi$ which is locally Lipschitz continuous in~$\R^2$ and such that $\H$ satisfies~\eqref{e.vi} with $\hat\psi$ in place of $\psi$. The Lipschitz continuity of $\H$ then follows from~\cite[Theorem 2(a)]{C}. It is clear from~\eqref{e.vi} or~\eqref{e.optp} that any $\hat\psi$ will do, provided that we modify $\psi$ only in the region in which $\{\H > \psi\}$ and in such a way that $\H \geq \hat \psi$ in $\R^2$. We define
\begin{equation*}\label{}
\hat\psi(y):= \begin{cases} h(y) & \mbox{if} \ y\in U^\delta\setminus U, \\
\psi(y) & \mbox{if} \ y \in U \cup (\R^2\setminus U^\delta),
\end{cases}
\end{equation*}
where $U^\delta:= \{ z\in \R^2\,:\, \dist(z,U ) < \delta\}$, $\delta >0$ is small enough that $\partial U^\delta$ is smooth and $\H > \psi$ in $U^\delta \setminus U$, and $h$ is the unique harmonic function in $U^\delta \setminus U$ which is equal to $\psi$ on $\partial (  U^\delta \setminus U)$. In view of the fact that $\H$ is superharmonic in $\R^2$, it is clear that $\hat\psi\leq \H$ in $\R^2$ and so $\psi$ has the desired properties. We deduce that $\H$ is locally Lipschitz. In particular, we have $\nabla \H \in L^\infty_{\mathrm{loc}}(\R^2;\R^2)$ by Rademacher's theorem (c.f.~\cite[Chapter 3]{EG}).

We next show that $\H$ has a Neumann trace on each side of $\partial U$. We mean that $\partial_\nu F$ exists and belongs to $L^2_{\mathrm{loc}}(\partial U)$, where $F$ is either the restriction of $\H$ to $\overline U$ or to 
$\R^2\setminus U$, and $\nu$ denotes the outer unit normal to $\partial U$. We only prove the claim in the case that $F = \H\vert_{\overline U}$, since the other case is even easier (due to the gap,~\eqref{e.gap}). We first give the argument in the case that $U$ is bounded. We write $\H\vert_{\overline U} = \H_1 + \H_2$ in $\overline U$, where $\H_1$ is harmonic in $U$ with $\H_1 = \H$ on $\partial U$, and $\H_2 := \H - \H_1$. Note that $\H \in W^{1,\infty}(\partial U)$ since $\H$ is Lipschitz. Next, we recall that the Dirichlet-to-Neumann map (also called the Poincar\'e-Steklov operator) with respect to the Laplacian is a continuous linear operator from~$H^{1}(\partial \Omega)$ to~$L^2(\partial \Omega)$, for any bounded $C^{1,1}$ domain~$\Omega$ (c.f.~\cite[Theorem~4.21]{M}). It follows that $\partial_\nu \H_1 \in L^2(\partial U)$. By standard elliptic regularity (c.f.~\cite[Theorem~9.15]{GT}), we have that $\H_2\in W^{2,p}(U)\cap W^{1,p}_0(U)$ for all $1<p <\infty$. Since $W^{2,p}(U) \hookrightarrow C^{1,\alpha}(\overline U)$ for $0<\alpha<1-1/p$ by 
Sobolev 
embedding (c.f.~\cite[Theorem 5.4,~Part II]{A}), we deduce that $\partial_\nu \H_2 \in C^\alpha(\partial U)$ for all $0<\alpha<1$. The claim that $\partial_\nu\left( \H\vert_{\overline U} \right) \in L^2(\partial U)$ follows. 

For the general case in which $U$ may be unbounded, we fix $R>0$ and select a $C^{1,1}$ domain $\widetilde U_R$ such that $U \cap B_{2R} \subseteq \widetilde U_R$ and apply the above argument to $\zeta_R \H$, where $\zeta_R$ is a smooth cutoff function satisfying $0\leq \zeta_R\leq1$, $\zeta_R \equiv 1$ on $B_R$, and $\zeta_R \equiv 0$ in $\Rd\setminus B_{2R}$.

We show next that $\muup_S$ is absolutely continuous with respect to the one-dimensional Hausdorff measure restricted to $\partial U$, with a density belonging to $L^2(\partial U)$. 
{Fix a compactly supported smooth function $\zeta \in C^\infty(\Rd)$. Using that $|\nabla \H| \in L^\infty_{\mathrm{loc}}(\R^2)$ and the fact proved in the previous paragraph that $\H$ has a Neumann trace in $L^2(\partial U)$ from both sides of $\partial U$, we may integrate by parts to find  
\begin{align*}
\int_{\R^2} \nabla \H(x) \cdot \nabla \zeta(x)\,dx & = \int_{\overline {U}}
\nabla \H(x)  \cdot \nabla \zeta(x)\,dx + \int_{\R^2 \backslash \overline{U}} 
\nabla \H(x) \cdot \nabla \zeta(x)\,dx \\
& =  - \int_{\overline U} \zeta(x) \Delta \H(x)\,dx+
 \int_{\partial U}  \zeta(x)\, \partial_\nu\left( \H\vert_{\overline U} \right) (x) \, d\mathcal H^1(x)\\& \qquad \qquad
- \int_{\R^2 \backslash \overline{U}} \zeta(x) \Delta \H (x)\, dx  -\int_{\partial U} \zeta(x)\, \partial_\nu\left( \H\vert_{\R^2
 \setminus U} \right)(x) \, d\mathcal H^1(x).
\end{align*}}It follows from this and~\eqref{deltah} that $\muup_S$ is equal to the one-dimensional Hausdorff measure, restricted to $\partial U$, multiplied by $L^2_{\mathrm{loc}}(\partial U)$ function $g:=\frac{1}{2\pi} \left( \partial_\nu\left( \H\vert_{\overline U} \right) -\partial_\nu\left( \H\vert_{\R^2\setminus U} \right)\right)$. Note that $g$ is nonnegative and cannot have mass greater than one, so $g\in L^2(\partial U)$.

Let $H^{\mu_0}$ be the potential generated by the measure $\mu_0$ given in~\eqref{mu0}. It is also the solution of the variational inequality~\eqref{e.vi} with $\psi$ replaced by $\frac12(c_3-|x|^2)$ for some $c_3\in \R$. The reason is that $\mu_0$ is the minimizer of $\I$ over all probability measures (without constraint) and hence the arguments of~Lemmas~\ref{lem2.1p=1} and~\ref{l.vi} apply.

Observe that
\begin{equation}\label{eqcoin}
  \left\{\H=\tfrac12 (c_1-|x|^2)\right\}= \overline{V} \cap\overline{ U} \quad \text{and} \ \left\{\H= \tfrac12 (c_2- |x|^2)\right\}= \overline{ (V \setminus  U)}.
\end{equation}
We show next that
\begin{equation}\label{chhmu}
H^{\mu_0} - \frac12 (c_3- c_2) \le \H \le H^{\mu_0} - \frac12 (c_3- c_1).
\end{equation}
We  prove only the second inequality, since the first is obtained via a similar argument. The function $G:=H^{\mu_0} - \frac12 (c_3- c_1)$ is larger than $\frac12(c_1-|x|^2)$ hence larger than $\psi$, and thus larger than $\H$ in $\{ \H = \psi\}$. Moreover, $G$ is superharmonic in $\R^2$ and $\H$ is harmonic in the complement of $\{ \H = \psi\}$, and thus $\H - G$ is subharmonic in $\R^2\setminus \{ \H=\psi\}$ and nonpositive on $\{ \H=\psi\}$. It follows from their definitions that $|\H - G|$ is bounded in $\R^2$. We deduce that $\H \leq G$ in $\R^2$, as desired.

Consider the difference of the two potentials, $w:= H^{\mu_0}- \H - \frac12 (c_3- c_1)$, which satisfies 
\begin{equation}
\label{subsup}
- \Delta w\ge 0 \ \text{in} \ \R^2\setminus \overline{V} \qquad \mbox{and} \qquad  -\Delta w \le 0 \ \text{in} \ \R^2\setminus \overline{D}
\end{equation}
and $0\le w\le \tfrac12 (c_1-c_2)$. We claim that:
\begin{equation}
\label{claim1}
\{w=0\}=\overline{ D} \cap\overline U,
\end{equation} 
\begin{multline}\label{claim2}
\mbox{either} \quad \{w=\tfrac12 (c_1-c_2)\} = 
\overline D \cap 
\overline{( V \setminus  U)}  \\  \mbox{or else} \quad \{w=\tfrac12 (c_1-c_2)\}=
(\overline D \cap \overline{( V \setminus  U)})\cup (\R^2\setminus D)
\end{multline}
and
\begin{equation}\label{claim3}
  U\not \subseteq D \qquad \mbox{implies that} \qquad \{w=\tfrac12 (c_1-c_2)\} = \overline D \cap \overline{( V \setminus  U)}.
\end{equation}

In order to prove these, we first observe that, by the monotonicity \eqref{chhmu}
of the obstacle problem,
\begin{equation}\label{inclu}
D\cap U \subseteq V\cap U \ \text{and } \ V \setminus U\subseteq  D \setminus U.
\end{equation}
Due to \eqref{eqcoin}, \eqref{inclu}, the fact that $H^{\mu_0} \geq \tfrac12(c_3- |x|^2)$ and $c_2<c_1$, we have
\begin{equation}
\label{inclu1}
\overline V \cap \{ w=0\}  \subseteq \overline  V \cap \overline  D  \cap \overline U = \overline D\cap \overline U \subseteq  \{w=0\}
\end{equation}
and
\begin{equation}
\label{inclu2}
\overline D \cap \left\{ w = \tfrac12 (c_1-c_2) \right\} \subseteq \overline D
\cap \overline {(V \setminus  U)}
= \overline {(V \setminus U)} \subseteq \left\{ w = \tfrac12 (c_1-c_2) \right\}.
\end{equation}

Here comes the argument for~\eqref{claim1}. In view of~\eqref{inclu1}, it suffices to show that $\{w=0\} \subseteq \overline V$. Suppose on the contrary that there exists $x\notin \overline{V}$ with $w(x) =0$. Then, by \eqref{subsup} and the strong maximum principle, we have $w \equiv 0$ in $\R^2 \setminus V$. By \eqref{inclu2}, continuity of $\H$  and $c_2\neq c_1$,  we deduce that $V \subseteq U$. In particular, $\H$ is harmonic in $\R^2\setminus \overline{U}$ and $w\equiv 0$ in $\R^2 \setminus\overline{U}$, which implies that $H^{\mu_0}$ is harmonic in $\R^2 \setminus \overline{U}$. This contradicts the assumption that $D \setminus U$ is nonempty, verifying~\eqref{claim1}.

We continue with the proofs of~\eqref{claim2} and~\eqref{claim3}.
In light of~\eqref{inclu2}, if $\{w=\tfrac12 (c_1-c_2)\}\neq \overline D 
\cap \overline{( V \setminus  U)}$, 
then there exists $x \notin D$ such that $w(x) =\tfrac12 (c_1-c_2)$. This implies that $w \equiv \tfrac12 (c_1-c_2)$ in $\R^2 \setminus D$ by the maximum principle and \eqref{subsup}. This confirms~\eqref{claim2}. Moreover, in this case $\H$ is harmonic outside of $\overline{D}$ (since $H^{\mu_0}$ is), which implies that $V\subseteq D$. Also, in view of \eqref{inclu2} and the continuity of $w$, we deduce that $U\subseteq D$, which completes the proof of~\eqref{claim3}.

We note that~\eqref{e.gap} follows from the combination of~\eqref{claim1},~\eqref{claim3} and~\eqref{inclu}.

We have left to prove~\eqref{e.sing},~\eqref{e.contract} and~\eqref{e.expand}. Each of these follows from a variation of an argument based on Hopf's lemma.

We first prove~\eqref{e.sing}, arguing by contradiction. Because of \eqref{inclu},
if \eqref{e.sing} fails then there
exists $x\in \overline V \cap \partial U \cap \overline D
\setminus \supp \muup_S$. Since $x \not\in \muup_S$, $\nabla w$ is continuous at $x$. In fact, we have $\nabla w(x) = 0$ since $w$ vanishes on $\overline U \cap \overline D$ by~\eqref{inclu1}. Due to~\eqref{subsup} and Hopf's lemma, we deduce that $w$ vanishes identically in the connected component of $\R^2\setminus V$ containing~$x$. Since $\H$ is 
harmonic in $\R^2 \setminus \overline V$ and $H^{\mu_0}$ is harmonic nowhere in $D$, we conclude that the intersection of $D$ and the connected component of $\R^2\setminus V$ containing $x$ is empty. This contradicts $x\in \overline D$, completing the argument for~\eqref{e.sing}.

Next we prove~\eqref{e.contract}. We need only verify the second claim, since the first claim was obtained above already in~\eqref{inclu}. In the case $U \not \subseteq D$, in view of \eqref{inclu}, it 
suffices to prove that $\overline{(V\setminus U)}
\cap \partial D$ is empty. Suppose on the contrary that there exists  
$x \in \overline{( V\setminus U)}
\cap \partial D$. We have already demonstrated a gap between 
$\overline {(V\setminus U)}$ and
$\partial U$ and thus $x\not \in \partial U$. In particular, $w\in C^{1,1}$ 
in a neighborhood of $x$. By \eqref{claim2}, $w\equiv \tfrac 12 (c_1-c_2)$ in 
$\overline D \cap \overline {(V \setminus U)}$ 
and thus $\nabla w(x)=0$. In view of Hopf's lemma and the fact that $w$ is subharmonic in $\R^2 \setminus D$ by~\eqref{subsup}, the second alternative must hold in~\eqref{claim2}. This contradicts~\eqref{claim3} to finish the proof of~\eqref{e.contract}.

We conclude with the argument for~\eqref{e.expand}, which is a bit more involved due to the fact that we need the best regularity for the boundary of the contact set in Caffarelli's theory, which is ``$C^{1,\alpha}$ except for singular points," see~\cite[Theorem 6]{C}. We proceed again by contradiction and suppose there exists $x \in\overline{U \cap D} \setminus ( V \cup \supp \muup_S)$. Since $x \in \overline D \setminus \supp\muup_S$, $x\not\in\partial U$ by~\eqref{e.sing} and so $w$ is $C^{1,1}$ in a neighborhood of $x$. Thus $w(x) = 0$ and $\nabla w(x)=0$. 
Notice that $x\in \partial V$ and thus it is a point on the boundary of the contact set. In view of \eqref{claim1} we have~$w>0$ in~$\R^2\setminus V$. Since $-\Delta w \ge 0 $ in~$\R^2\setminus V$ by~\eqref{subsup}, an application of Hopf's lemma gives us a contradiction, completing the proof. However, before we may apply Hopf's lemma, we must check that $\partial V$ is sufficiently smooth in a neighborhood of~$x$. Indeed, Hopf's lemma holds for $C^{1,\alpha}$ boundaries 
{(for any $\alpha > 0$)} 
but not for domains with boundaries which are merely $C^1$, see~Safonov~\cite{Sa}. 

Fortunately, the free boundary regularity theory of Caffarelli gives us just what we need: $\partial V$ is $C^{1,\alpha}$ near~$x$, 
{for every $0<\alpha < 1$.} 
The reason is that $x$~cannot be a singular point (as defined in~\cite{C}) of the free boundary, due to the fact that $\partial (U\cap D)$ is Lipschitz and $U \cap D \subseteq V$, which rules out the possibility that~$\partial V$ is contained in a ``thin strip" near $x$.  See~\cite[Theorem 6]{C}, which asserts that the boundary of~$V$ is $C^{1,\alpha}$ near~$x$ unless $V \cap B_r(x)$ is contained between two parallel planes separated by a distance of $o(r)$ as $r\to 0$. It follows that~$\partial V$ is regular enough in a neighborhood of $x$ to invoke Hopf's lemma. This completes the proof.
\end{proof}

\section{An example: real part constraint}
\label{sec-example}

We conclude with an explicit example, in which the constraint set $U$ is a half-space and $p=1$. We show that there is a critical point at 
which the minimizing measure becomes entirely concentrated 
(and equal to the semicircle law) on the boundary line.

For each $a\in \R$, we set $U_a := \{ z\in \C \,: \, \Re(z)< -a\}$, 
and set~$\mathcal C_a:=\{\mu\in \mathcal P(\C): \mu(\overline{U}_a)=1\}$. We 
denote 
by $L_a=\partial U_a$ 
the boundary line, 
by $\muup_a$ the minimizing measure of $I(\cdot)$ on $\mathcal C_a$, and 
by
$\muup_{S,a}$ its singular component as in Theorem \ref{T}. Note that $\muup_{S,a}(L_a)$ is the total mass of $\muup_{S,a}$. Also denote the semicircle law on $\R$ by $$\sigma(dy)=\frac{\sqrt{2-y^2}}{\pi} {\bf 1}_{\{|y|
    <\sqrt{2}\}} \, dy. $$

The following result asserts that $\muup_a = \muup_{S,a}$ if and only if $a\geq \sqrt{2}$, in which case $\muup_{S,a}$ is the semicircle law on~$L_a$.

\begin{prop}
  \label{prop-example}

For every $a<\sqrt{2}$, we have $\muup_{S,a}(L_a)<1$. Conversely, if $a\geq \sqrt{2}$, then $\muup_{S,a}(L_a)=1$ and, with $z=x+iy$, 
  $$\muup_{S,a}(dz)=\delta_{a}(dx)\times \sigma(dy)=:\sigma_{a}(dz)\,.$$
  \end{prop}
  
  Before giving the proof of Proposition \ref{prop-example}, we recall some
  preliminaries on the semicircle law. Denote 
    \begin{equation*}\label{}
H^\sigma(z) : = -\int_{\R} \log|z-iy| \, \sigma(d y).
\end{equation*}
By a direct computation, we have
  \begin{eqnarray}
   2H^\sigma(x)+x^2-(
	1+(\log 2)/2)
    && \!\!\!\!\!\!
    \left\{
      \begin{array}{ll}
	=0,&
	|x|\leq \sqrt{2}\\
	\geq 0,& 
	|x|\leq \sqrt{2}\end{array}
	\right., \ x\in \R,\quad 
	\mbox{(\cite[Ex. 2.6.4]{AGZ})}.
	\label{log-pot}\\
	\mathcal S_\sigma(z):=\int\frac{1}{x-z} \sigma(dx)&=& -(z-\sqrt{z^2-2}),
	\ z\in \C_+\,,
	\quad 
	\mbox{(\cite[pg. 46, (2.4.7)]{AGZ})}.
	\label{S-transf}
      \end{eqnarray}
      (We recognize $\mathcal S_\sigma
      (z)$ as the Stieltjes transform of $\sigma$.)
    
    \begin{proof}[Proof of Proposition~\ref{prop-example}]
    Note first that if $\muup_{S,a}(\overline{U}_a)=1$ then 
      $\muup_{S,a}(L_a)=1$ by Theorem \ref{T} and therefore, using that
      for $z=a+iy$ one has $|z|^2=a^2+y^2$, it follows from \eqref{log-pot}
      that
      necessarily in such a situation
      $\muup_{S,a}(a+idy)=\sigma(dy)$ and 
      \begin{equation}
	\label{eq-5o}
	2H^{\muup_{S,a}}(a+iy)+|(a+iy)|^2=a^2+1+\frac{\log 2}{2}.
      \end{equation}
      Applying Lemma \ref{lem2.1p=1}, we conclude that
      a necessary and sufficient condition for~$\muup_{S,a}(L_a)=1$ is that
      \begin{equation}
	\label{eq-6o}
	2H^{\sigma_{a}}(b+iy)+|(b+iy)|^2\geq a^2+1+\frac{\log 2}{2} \quad
	\mbox{for all} \ b>a.
      \end{equation}
Suppose $a<\sqrt{2}$. Differentiate the left side of \eqref{eq-6o}
      with respect to $b$  to obtain, with $\theta\in \R$,
      \begin{eqnarray*}
	\frac{\partial}{\partial b}
	\left(2H^{\sigma_{a}}(b+iy)+|(b+iy)|^2i\right)
&=&2b-2\Im  \mathcal S_\sigma(y+i(b-a))\nonumber\\
&=&2b+2\Im (z-\sqrt{z^2-2})=:F(a,b,y)\,,
\end{eqnarray*}
where $z=y+i(b-a)$ and \eqref{S-transf} was used in the equalities.
Noting that 
$$\lim_{b\searrow a} F(a,b,0)=2a-2\sqrt{2}\,,$$
we conclude that if $a<\sqrt{2}$ the necessary condition for
$\muup_{S,a}(L_a)=1$ is not satisfied. This proves the first statement of the proposition. 

For the second statement, we introduce the function
$$G(y,a,b)=y^2+b^2+2H^\sigma(y+i(b-a))-a^2-1-\frac{\log 2}{2}.$$
Since $G(y,a,a)\geq 0$ by \eqref{log-pot}, it is enough due to Lemma \ref{lem2.1p=1} to 
verify that $G(y,a,b)\geq 0$ when $b>a$.
By symmetry and convexity, 
for fixed $a<b$, one has $G(y,a,b)\geq G(0,a,b)=:\overline G(a,b)$. 
Note that 
$$\frac{\partial \overline G(a,b)}{\partial b}
 \left\{\begin{array}{ll}
=  2\big(2b-a-\sqrt{2-(b-a)^2}\big),& 0\leq b-a<\sqrt{2}\,,
\\
>0,& b-a\geq \sqrt{2}\,,
\end{array}\right.$$
where again we used \eqref{S-transf}.
It is straightforward to check that the right side is nonnegative provided that
$b\geq a$ and $a\geq \sqrt{2}$. Since $\overline G(a,a)\geq 0$, this implies that
$\overline G(a,b)\geq 0$ for $b>a$. This completes the proof.
\end{proof}

\bibliographystyle{plain}
\def\cprime{$'$}

\end{document}